\documentclass[a4paper]{article}
\usepackage{amssymb}

\def\pushright#1{{
   \parfillskip=0pt            
   \widowpenalty=10000         
   \displaywidowpenalty=10000  
   \finalhyphendemerits=0      
   \leavevmode                 
   \unskip                     
   \nobreak                    
   \hfil                       
   \penalty50                  
   \hskip.2em                  
   \null                       
   \hfill                      
   {#1}                        
   \par}}                      

\def\qed{\pushright{$\square$}\penalty-700 \smallskip}

\newtheorem{theorem}{Theorem}[section]
\newtheorem{lemma}[theorem]{Lemma}
\newtheorem{corollary}[theorem]{Corollary}

\newenvironment{proof}%
 {\begin{trivlist}\item[]{\bf Proof }}%
 {\qed\end{trivlist}}

\newcommand{\Bb}
{\mathbb}

\newcommand{\llim}
 {{\mbox{$\lower.95ex\hbox{{\rm lim}}$}\atop{\scriptstyle{\leftarrow}}}{}}

\newcommand{\rlim}
 {{\mbox{$\lower.95ex\hbox{{\rm lim}}$}\atop{\scriptstyle{\rightarrow}}}{}}

\newcommand{\ob}
 {{\rm ob}}

\newcommand{\op}
 {^{\rm op}}

\newcommand{\prarr}
 {\rightrightarrows}

\newcommand{\Set}
 {{\bf Set}}

\newcommand{\stp}[3]
 {\mbox{$#1\! : #2 \to #3$}}

\begin{document}

\title{Notes on Commutation of Limits and Colimits}
\author{Marie Bjerrum \\ {\it Department of Pure Mathematics, University of Cambridge, England} \\
Peter Johnstone \\ {\it Department of Pure Mathematics, University of Cambridge, England} \\
Tom Leinster \\ {\it School of Mathematics, University of Edinburgh, Scotland} \\
and William F.~Sawin%
\thanks{Supported by the National Science Foundation Graduate Research
  Fellowship under Grant No.\ DGE-1148900}%
\\ {\it Department of Mathematics, Princeton University, U.S.A.}}
\date{\today}

\maketitle

\begin{abstract}
We show that there are infinitely many distinct closed classes of colimits (in the sense of the Galois connection induced by commutation of limits and colimits in $\Set$) which are intermediate between the class of pseudo-filtered colimits and that of all (small) colimits. On the other hand, if the corresponding class of limits contains either pullbacks or equalizers, then the class of colimits is contained in that of pseudo-filtered colimits.
\end{abstract}

\section{Introduction}

The fact that colimits over filtered categories commute with finite limits in the category of sets (and in many related categories) is very well known, and most students encounter it as part of a first course in category theory (cf.~\cite{CWM}, section IX 2). In what follows we shall restrict ourselves to limits and colimits in $\Set$. The relation `limits over $\Bb I$ commute with colimits over $\Bb J$' gives rise to a Galois connection between classes of small categories regarded as limit-shapes and the same classes regarded as colimit-shapes: if we set
\[{\cal I}^r =\{{\Bb J}\mid \mbox{$\rlim_{\Bb J}$ commutes with $\llim_{\Bb I}$ for all }{\Bb I}\in {\cal I}\}\]
and
\[{\cal J}^l =\{{\Bb I}\mid \mbox{$\rlim_{\Bb J}$ commutes with $\llim_{\Bb I}$ for all }{\Bb J}\in {\cal J}\}\]
then we say $\cal I$ (resp.\ $\cal J$) is a {\em closed class\/} of limits (resp.\ colimits) if ${\cal I}={\cal I}^{rl}$ (resp.\ ${\cal J}={\cal J}^{lr}$). It is also well known that the class of filtered colimits is closed in this sense (filtered colimits are precisely those which commute with all finite limits --- cf.~Lemma 2.1 below), but the class of finite limits is not: any category admitting an initial functor from a finite category (in the sense of \cite{CWM}, p.~218) is in the closure of the latter class.

There is a proper class of distinct closed classes: for any regular cardinal $\kappa$, the class of $\kappa$-filtered colimits is closed (being the class of colimits which commute with all $\kappa$-small limits), and these classes are all distinct. Nevertheless, it seems not unreasonable to hope that we might be able to classify all the closed classes of colimits which contain that of filtered colimits (equivalently, for which the corresponding closed class of limits is generated by its finite members). There are seven such classes which are more or less well known: in addition to filtered colimits and all colimits, we have the class of pseudo-filtered colimits (also called weakly filtered colimits) which are precisely those colimits commuting with all finite connected limits; a category is pseudo-filtered iff each of its connected components is filtered. The class of sifted colimits (`colimites tamisantes' in French \cite{lair}), being the class of colimits which commute with all finite products, is also reasonably well known by now. Three more classes arise from the anomalous behaviour of the empty set (the empty colimit does not commute with the empty limit, but does commute with all nonempty ones): we may add the empty colimit to the class of filtered colimits or the class of sifted colimits, to obtain the classes which commute with all nonempty finite limits (resp.\ with nonempty finite products), and by cutting down the class of all colimits to the class of connected colimits, we obtain the class which commutes (with all conical limits and) with the empty limit.

Incidentally, we say a category $\Bb I$ is {\em conical\/} if the identity functor ${\Bb I}\to{\Bb I}$ has a cone over it. It is easy to see that this is equivalent to saying that $\Bb I$ {\em has an initial idempotent\/} in either of two possible senses: ($a$) that the idempotent-completion of $\Bb I$ has an initial object, and ($b$) that there is an initial functor ${\Bb E}\to{\Bb I}$ where $\Bb E$ is the `free living idempotent', i.e.\ the two-element monoid $\{1,e\}$ with $e^2=e$. It follows easily that conical limits are absolute (that is, preserved by all functors), and that they are precisely the limits which commute with all colimits in $\Set$ (again, this follows from Lemma 2.1 below, taking ${\Bb J}={\Bb I}\op$).

It is tempting to conjecture that these seven classes might be the only ones containing all filtered colimits, and indeed the first author of the present paper claimed at one stage to have a proof of this. However, the conjecture is false, and the counterexamples are surprisingly easy to describe. An infinite family of counterexamples was discovered by the second author in December 2013, and improved by the third author the following day, when he showed that they could be reduced to commutation of limits and colimits over groups of coprime orders, considered as categories with one object. The fourth author, in response to a query by the third on the MathOverflow website \cite{sawin}, provided a necessary and sufficient condition for limits over one group to commute with colimits over another. The purpose of this note is to present both (some of) the evidence in favour of the conjecture accumulated by the first author, and the group-theoretic counterexamples; we do not claim to be anywhere near a complete classification of the closed classes of colimits containing filtered colimits, and indeed we suspect that such a classification would be very hard to obtain.

\section{Colimits commuting with pullbacks or equalizers}

If one considers the Hasse diagram of the seven closed classes of colimits mentioned in the Introduction, it soon becomes clear that the only interval where one might expect to find further closed classes is that between (pseudo-filtered colimits) and (all colimits). For example, there can be no closed class intermediate between (filtered or empty colimits) and (pseudo-filtered colimits), because as soon as one adds a single non-connected category to the list of limit shapes with which $\Bb J$-colimits must commute, one forces $\Bb J$ to have at most one connected component. In what follows, we therefore concentrate our attention on this interval.

The following necessary condition for commutation of limits and colimits is well-known (and has already been invoked in the Introduction).

\begin{lemma}
Suppose $\Bb I$-limits commute with $\Bb J$-colimits in $\Set$. Then $\Bb J$ has cocones over all diagrams of shape ${\Bb I}\op$.
\end{lemma}

\begin{proof}
Suppose given such a diagram \stp{D}{{\Bb I}\op}{{\Bb J}}. Consider the functor \stp{F}{{\Bb I}\times{\Bb J}}{\Set} sending $(i,j)$ to ${\Bb J}\,(Di,j)$; it is easy to see that $\rlim_{\Bb J}F(i,-)\cong 1$ for all $i$, so that $\llim_{\Bb I}\rlim_{\Bb J}F\cong 1$. But elements of $\llim_{\Bb I}F(-,j)$ correspond to cocones over $D$ with vertex $j$; so if there are no such cocones then $\rlim_{\Bb J}\llim_{\Bb I}F$ is empty.
\end{proof}

The next lemma is much less well-known; it is due to F.~Foltz~\cite{foltz}.

\begin{lemma}
Suppose that $\Bb I$-limits commute with $\Bb J$-colimits in $\Set$, and that $\Bb I$ is connected but not conical. Then $\Bb J$ has cocones over diagrams of shape $(\bullet\leftarrow\bullet\to\bullet)$.
\end{lemma}

\begin{proof}
Suppose not; let $(j_1\stackrel{\beta}{\leftarrow}j_0\stackrel{\gamma}{\to} j_2)$ be a diagram of the indicated shape with no cocone over it. Consider the functor \stp{F}{{\Bb I}\times{\Bb J}}{\Set} defined by
\[F(-,-) = \left.\left(\left(\coprod_{i\in\ob\ {\Bb I}}{\Bb I}\,(i,-)\right)\times {\Bb J}\,(j_0,-)\right)\right/\sim\]
where $\sim$ is the equivalence relation (clearly a congruence) which identifies all pairs $(\alpha,\delta),(\alpha',\delta)$ for which $\alpha$ and $\alpha'$ have the same codomain and $\delta$ factors through either $\beta$ or $\gamma$. It is easy to see that, for fixed $j$, $F(-,j)$ may be written as 
\[\coprod_{\delta\in A}\left(\coprod_{i\in\ob\ {\Bb I}}{\Bb I}\,(i,-)\right)\amalg\coprod_{\delta\in B}1\]
where $A$ is the set of maps \stp{\delta}{j_0}{j} which do not factor through $\beta$ or $\gamma$, and $B$ is the set of those which do; but $\llim_{\Bb I}$ preserves coproducts since $\Bb I$ is connected, and $\llim_{\Bb I}{\Bb I}\,(i,-)=\emptyset$ since $\Bb I$ is not conical, so $\llim_{\Bb I}F(-,j)\cong B$. In other words, $\llim_{\Bb I}F$ is the subfunctor of ${\Bb J}\,(j_0,-)$ which is the union of the images of ${\Bb J}\,(\beta,-)$ and ${\Bb J}\,(\gamma,-)$; and the latter union is disjoint since there is no cocone over $(\beta,\gamma)$. It follows easily that $\rlim_{\Bb J}\llim_{\Bb I}F$ has two elements. On the other hand, it is not hard to see that $\rlim_{\Bb J}F(i,-)$ has just one element for any $i$, since any two elements $(\alpha,\delta)$ and $(\alpha',\delta')$ can be linked by a zigzag of the form
\[(\alpha,\delta)\leftarrow(\alpha,1_{j_0})\to(\alpha,\beta)\sim(\alpha',\beta)\leftarrow (\alpha',1_{j_0})\to(\alpha',\delta')\ .\]
So $\llim_{\Bb I}\rlim_{\Bb J}F$ is a singleton, which provides the required contradiction.
\end{proof}

\begin{corollary}
If colimits of shape $\Bb J$ commute with equalizers in $\Set$, then $\Bb J$ is pseudo-filtered.
\end{corollary}

\begin{proof}
By the two previous lemmas, $\Bb J$ has cocones over diagrams of shapes $(\bullet\prarr\bullet)$ and $(\bullet\leftarrow\bullet\to\bullet)$; but these suffice for pseudo-filteredness. 
\end{proof}

For completeness, we also record

\begin{lemma}
If colimits of shape $\Bb J$ commute with pullbacks in $\Set$, then $\Bb J$ is pseudo-filtered.
\end{lemma}

\begin{proof}
This follows from the well-known fact (\cite{elephant}, A1.2.9) that if $\cal C$ is a category with all finite limits then any functor ${\cal C}\to{\cal D}$ which preserves pullbacks preserves all finite connected limits.
\end{proof}

Since both pullbacks and equalizers are required to generate the class of finite connected limits (in the sense that a category has all finite connected limits iff it has pullbacks and equalizers, though it can have either pullbacks or equalizers without possessing the other), the fact that either pullbacks or equalizers suffice to generate the closure of finite connected limits (in the sense of the Galois connection mentioned in the Introduction) is striking.

\section{Some new classes of commuting limits and colimits}

By Foltz's Lemma 2.2, if we wish to find a closed class of colimits lying between (pseudo-filtered colimits) and (all small colimits) (equivalently, a closed class of limits lying between (conical limits) and the closure of (finite connected limits)), we must consider colimits over categories which have cocones over diagrams of shape $(\bullet\leftarrow\bullet\to\bullet)$, but not over those of shape $(\bullet\prarr\bullet)$. In trying to construct such a category, one might be led to the following example: the objects of $\Bb J$ are the natural numbers, there is one (identity) morphism $n\to n$ for each $n$, and two morphisms $n\prarr m$ (labelled $0$ and $1$) for each $n<m$, and composition of non-identity morphisms is defined by addition of labels mod $2$, i.e.\ the composite of two morphisms with the same label is labelled $0$ and that of two morphisms with different labels is labelled $1$. And indeed, there are non-conical limits which commute with colimits over this category; specifically, limits over the cyclic group of order 3 (or, more generally, any nontrivial group of odd order) considered as a category. (More generally still, if we modify $\Bb J$ so that it has $k>1$ morphisms $n\to m$ whenever $n<m$, with labels $0$ to $k-1$ and composition given by addition mod $k$, then colimits over $\Bb J$ commute with limits over any group of order prime to $k$.)

But it is not necessary for the colimit category $\Bb J$ to have more than one object; we can take it to be a group, too. Of course, any group satisfies the condition that it has cocones over diagrams of shape $(\bullet\leftarrow\bullet\to\bullet)$ (semigroup-theorists call this the left Ore condition), but no nontrivial group has cocones over diagrams of shape $(\bullet\prarr\bullet)$. We may now prove

\begin{lemma}
Let $G$ and $H$ be finite groups of coprime orders, considered as categories. Then limits over $G$ commute with colimits over $H$ in $\Set$.
\end{lemma}

\begin{proof}
Of course, a functor $G\to\Set$ is just a set equipped with an action of $G$; its limit is its set of $G$-fixed points, and its colimit is the set of $G$-orbits. So we need to show that, if $G\times H$ acts on a set $A$, then the $H$-orbits which are fixed under the induced action of $G$ on these orbits are exactly the orbits which contain (equivalently, consist of) $G$-fixed elements of $A$. But this is easy: if $a$ is any element of such a fixed orbit, then for any $g\in G$ we have $(g,1)a=(1,h)a$ for some $h\in H$, and then since $(g,1)$ and $(1,h)$ commute we have $(g^n,1)a=(1,h^n)a$ for all $n$. In particular, the least $n>0$ such that $(g^n,1)a=a$ must divide both the order of $g$ and that of $h$; but these two are coprime, and so $(g,1)a=a$.
\end{proof}

In fact, by an (only slightly) more sophisticated argument, we may obtain a necessary and sufficient condition on a pair of groups $G,H$ for limits over $G$ to commute with colimits over $H$.

\begin{lemma}
Given two groups $G$ and $H$, limits over $G$ commute with colimits over $H$ in $\Set$ iff no nontrivial quotient group of $G$ is isomorphic to a subquotient group of $H$.
\end{lemma}

\begin{proof}
Since any $(G\times H)$-set is the coproduct of its orbits (and $\llim_G$ and $\rlim_H$ both preserve coproducts), it suffices to consider a single $(G\times H)$-orbit, i.e.\ a $(G\times H)$-set isomorphic to the set $(G\times H):S$ of left cosets of a subgroup $S$ of $G\times H$. By Goursat's Lemma~\cite{lambek}, any subgroup $S \leq G\times H$ is of the form $\{(g,h)\in K_1\times K_2\mid \theta(gL_1)=hL_2\}$ where $L_1\lhd K_1\leq G$, $L_2\lhd K_2\leq H$ and \stp{\theta}{K_1/L_1}{K_2/L_2} is an isomorphism. It is easy to see that the $H$-orbits of $(G\times H):S$ form a single $G$-orbit, which may be identified with $G:K_1$; so there is a $G$-fixed $H$-orbit iff $K_1=G$. On the other hand, $(G\times H):S$ has $G$-fixed points iff $L_1=G$ (in which case $S$ is simply $G\times L_2$); so the given condition is equivalent to saying that every transitive $(G\times H)$-set has either a $G$-fixed point or no $G$-fixed $H$-orbits.
\end{proof}

In particular, we may conclude that there are infinitely many closed classes of colimits intermediate between pseudo-filtered colimits and all small colimits: for each finite simple group $G$, the class of colimits which commute with limits over $G$ contains colimits over all groups which do not have $G$ as a subquotient, but not over $G$ itself. Thus the problem of classifying all closed classes of colimits which contain all pseudo-filtered colimits is at least as hard as that of classifying finite simple groups.


\begin{thebibliography}{9}
\bibitem{foltz} F.~Foltz, Sur la commutation des limites, {\it Diagrammes\/} 5 (1981).
\bibitem{elephant} P.T.~Johnstone, {\it Sketches of an Elephant: a Topos Theory Compendium}, vol. 1, Oxford Logic Guides no.~43 (Oxford University Press, 2002).
\bibitem{lair} C.~Lair, Sur le genre d'esquissabilit\'e des cat\'egories modelables (accessibles) poss\'edant les produits de deux, {\it Diagrammes\/} 35 (1996), 25--52.
\bibitem{lambek} J.~Lambek, Goursat's theorem and the Zassenhaus lemma, {\it Canad.\ J.\ Math}. 10 (1958), 45--56.
\bibitem{CWM} S.~Mac~Lane, {\it Categories for the Working Mathematician}, Graduate Texts in Math.\ no.\ 5 (Springer--Verlag, 1971; revised edition 1998).
\bibitem{sawin} {\tt http://mathoverflow.net/questions/152193} (2013). 
\end{thebibliography}
\end{document}